\documentclass[11pt,a4paper,reqno]{amsart}

\usepackage[T1]{fontenc}
\usepackage[english]{babel}

\usepackage{mathtools}
\usepackage{amssymb}
\usepackage{libertinus}
\usepackage[cal=boondoxo,bb=ams]{mathalfa}
\usepackage{microtype}
\usepackage{xcolor}

\allowdisplaybreaks[2]
\usepackage[
a4paper,
left=30mm,
right=30mm,
top=27mm,
bottom=30mm,
headsep=8mm,
footskip=13mm,
heightrounded
]{geometry}

\usepackage{enumitem}
\usepackage{array,booktabs,tabularx}

\setlist[itemize]{
	leftmargin=2.1em,
	itemsep=0.25em,
	topsep=0.45em,
	parsep=0pt
}

\setlist[enumerate]{
	leftmargin=2.1em,
	itemsep=0.25em,
	topsep=0.45em,
	parsep=0pt
}

\numberwithin{equation}{section}

\theoremstyle{plain}
\newtheorem{theorem}{Theorem}[section]
\newtheorem{lemma}[theorem]{Lemma}
\newtheorem{prop}[theorem]{Proposition}
\newtheorem{coro}[theorem]{Corollary}

\theoremstyle{definition}
\newtheorem{defn}[theorem]{Definition}

\theoremstyle{remark}
\newtheorem{remark}[theorem]{Remark}
\newtheorem{exam}[theorem]{Example}

\makeatletter

\def\@settitle{%
	\begin{center}
		\vspace*{-0.4em}
		{\normalfont\bfseries
			\fontsize{17}{21}\selectfont
			\@title\par}
		\vspace{0.7em}
	\end{center}
}

\renewcommand\section{%
	\@startsection{section}{1}{\z@}%
	{1.5\baselineskip plus 0.2\baselineskip minus 0.1\baselineskip}%
	{0.65\baselineskip}%
	{\normalfont\Large\bfseries}%
}

\renewcommand\subsection{%
	\@startsection{subsection}{2}{\z@}%
	{1.15\baselineskip plus 0.2\baselineskip minus 0.1\baselineskip}%
	{0.45\baselineskip}%
	{\normalfont\large\bfseries}%
}

\renewcommand\subsubsection{%
	\@startsection{subsubsection}{3}{\z@}%
	{0.9\baselineskip plus 0.15\baselineskip minus 0.1\baselineskip}%
	{0.35\baselineskip}%
	{\normalfont\normalsize\bfseries}%
}

\makeatother

\usepackage{etoolbox}

\makeatletter

\patchcmd{\@setauthors}
{\centering\footnotesize}
{\centering\normalsize}
{}
{\PackageWarning{Time-dependent-template}
	{Failed to patch the author font size}}

\patchcmd{\@setauthors}
{\MakeUppercase{\authors}}
{\authors}
{}
{\PackageWarning{Time-dependent-template}
	{Failed to patch the author capitalization}}

\patchcmd{\maketitle}
{\uppercasenonmath\shorttitle}
{}
{}
{\PackageWarning{Time-dependent-template}
	{Failed to patch the running-title formatting}}

\patchcmd{\maketitle}
{\@nx\MakeUppercase{\the\toks@}}
{\the\toks@}
{}
{\PackageWarning{Time-dependent-template}
	{Failed to patch the running-author formatting}}

\makeatother

\usepackage{hyperref}

\hypersetup{
	hidelinks,
	pdfencoding=auto,
	pdftitle={Blow-up criteria and lifespan estimates for semilinear wave equations with time-dependent damping and mass},
	pdfauthor={Wenhui Chen and Mohamed Ali Hamza},
	pdfsubject={Blow-up and lifespan estimates for semilinear wave equations with time-dependent lower-order terms},
	pdfkeywords={semilinear wave equation, time-dependent damping, time-dependent mass, derivative nonlinearity, blow-up, lifespan, positive adjoint solution, Volterra equation}
}

\usepackage[nameinlink,capitalise,noabbrev]{cleveref}

\crefname{theorem}{Theorem}{Theorems}
\crefname{lemma}{Lemma}{Lemmas}
\crefname{prop}{Proposition}{Propositions}
\crefname{coro}{Corollary}{Corollaries}
\crefname{defn}{Definition}{Definitions}
\crefname{remark}{Remark}{Remarks}
\crefname{exam}{Example}{Examples}
\crefname{section}{Section}{Sections}
\crefname{subsection}{Subsection}{Subsections}

\newtheorem{assumption}{Assumption}[section]
\newcommand{\ml}{\mathcal}
\newcommand{\mb}{\mathbb}
\newcommand{\dd}{\,\mathrm{d}}
\newcommand{\Rplus}{\mb{R}_+}
\newcommand{\Rplusast}{\mb{R}_+^\ast}

\DeclareMathOperator{\supp}{supp}

\title[Wave equations with time-dependent damping and mass]
{Blow-up criteria and lifespan estimates for semilinear wave equations
with time-dependent damping and mass}

\author[W. Chen and M. A. Hamza]{Wenhui Chen and Mohamed Ali Hamza}

\address{Wenhui Chen: School of Mathematics and Information Science,
Guangzhou University, Guangzhou 510006, P. R. China}
\email{wenhui.chen.math@gmail.com}

\address{Mohamed Ali Hamza: Department of Basic Sciences,
Deanship of Preparatory Year and Supporting Studies,
Imam Abdulrahman Bin Faisal University,
P. O. Box 1982, Dammam, Saudi Arabia}
\email{mahamza@iau.edu.sa}

\keywords{
	semilinear wave equation,
	time-dependent damping,
	time-dependent mass,
	derivative-type nonlinearity,
	blow-up,
	lifespan estimate}

\subjclass[2020]{
	Primary 35L71;
	Secondary 35B44, 35L05}

\date{}

\begin{document}

\begin{abstract}
	We consider semilinear wave equations with time-dependent damping and mass and a derivative-type nonlinearity. By applying a Liouville transformation and solving a Volterra integral equation posed from infinity, we construct a positive exact solution to the adjoint equation without using an explicit representation of the linear propagator. This solution is used to derive a blow-up criterion for energy solutions with finite propagation, together with an upper bound for the lifespan. If the primitive of the damping coefficient grows at most logarithmically, we obtain the full shifted Glassey range, including the critical exponent, and the corresponding polynomial and exponential lifespan estimates. The result applies independently of the sign of the scale-invariant discriminant and therefore includes the mass-dominant regime. It also covers non-integrable oscillatory perturbations of the damping coefficient when their contributions are canceled by the corresponding mass terms.
\end{abstract}

\maketitle

\section{Introduction}\label{Section-Introduction}

We consider the following Cauchy problem for a semilinear wave equation with time-dependent damping and mass and a derivative-type nonlinearity:
\begin{align}\label{Main-Problem}
	\begin{cases}
		u_{tt}-\Delta u+b(t)u_t+m(t)u=|u_t|^p,&x\in\mb{R}^n,\ t\in\Rplusast,\\
		u(0,x)=\varepsilon u_0(x),\ \ u_t(0,x)=\varepsilon u_1(x),&x\in\mb{R}^n,
	\end{cases}
\end{align}
where $n\geqslant1$ denotes the space dimension, $p>1$, and $\varepsilon>0$ measures the size of the initial data. The functions $b=b(t)$ and $m=m(t)$ represent the time-dependent damping and mass coefficients, respectively. To formulate the assumptions on these coefficients and the blow-up criterion below, we introduce
\begin{align}\label{Definition-BV}
	B(t):=\int_0^t b(\tau)\,\mathrm{d}\tau\ \ \text{and}\ \ 	V(t):=m(t)-\frac12 b'(t)-\frac14[b(t)]^2.
\end{align}
\begin{assumption}[Coefficient assumptions]\label{Assumption-Coefficients}
	The damping and mass coefficients satisfy
	\begin{align*}
		b\in \mathcal{C}^1(\Rplus)\cap L^\infty(\Rplusast),	\ \ m\in \mathcal{C}(\Rplus),	\ \ V\in L^1(\Rplusast).
	\end{align*}
\end{assumption}

\noindent Under Assumption~\ref{Assumption-Coefficients}, we derive a blow-up criterion for \eqref{Main-Problem} and an upper bound for the lifespan.

\subsection{Wave equations with derivative-type nonlinearities}
Let us first recall the corresponding problem without lower-order terms,
\begin{align}\label{Classical-Derivative-Wave}
	u_{tt}-\Delta u=|u_t|^p.
\end{align}
For every real number $d>1$, we define the Glassey exponent by
\begin{align*}
	p_{\mathrm{Gla}}(d):=1+\frac{2}{d-1},
\end{align*}
and put $p_{\mathrm{Gla}}(1):=\infty$. For the classical problem \eqref{Classical-Derivative-Wave} in $\mb{R}^n$, finite-time blow-up in the range $1<p\leqslant p_{\mathrm{Gla}}(n)$ under suitable positivity assumptions on the initial data, as well as complementary small data global in-time existence results in the supercritical range, have been established in a series of works (see \cite{Glassey=1981,Sideris=1983,Rammaha=1987,Hidano-Tsutaya=1995,Zhou=2001,Hidano-Wang-Yokoyama=2012,Ikeda-Sobajima-Wakasa=2019} and references therein). The corresponding upper bounds for the lifespan are
\begin{align*}
	T_\varepsilon\leqslant
	\begin{cases}
		C\varepsilon^{-\frac{2(p-1)}{2-(n-1)(p-1)}}&\text{if}\ \ 1<p<p_{\mathrm{Gla}}(n),\\
		\exp\big(C\varepsilon^{-(p-1)}\big)&\text{if}\ \ p=p_{\mathrm{Gla}}(n).
	\end{cases}
\end{align*}

The presence of a time-dependent damping term may or may not change this picture, depending on its long-time behavior. At the linear level, a systematic distinction between non-effective and effective time-dependent dissipation was developed in \cite{Wirth=2006,Wirth=2007}. Roughly speaking, dissipation below the scaling level preserves a wave-like asymptotic structure, whereas dissipation above the scaling level produces a diffusive behavior.

 In the massless case $m\equiv0$ but $b(t)\not\equiv0$ in \eqref{Main-Problem}, the authors of \cite{Lai-Takamura=2019} considered nonnegative weak damping coefficients. In particular, when $b\in L^1(\Rplusast)$, they recovered the classical Glassey blow-up range and the same lifespan upper bounds as for the undamped equation. Thus, an integrable damping coefficient remains short-range from the viewpoint of the blow-up threshold generated by the derivative-type nonlinearity.

A different behavior occurs for the scale-invariant damping
\begin{align*}
	b(t)=\frac{\mu}{1+t}\ \ \mbox{and}\ \ m(t)\equiv0.
\end{align*}
The corresponding linear equation was studied in \cite{Wirth=2004}, where an explicit representation in terms of special functions and some estimates were established. In particular, the analysis shows that the long-time
behavior depends essentially on the size of $\mu$. The first direct blow-up and lifespan estimates for the equation with the nonlinearity $|u_t|^p$ were obtained by \cite{Lai-Takamura=2019}. Their result covers the range $1<p\leqslant p_{\mathrm{Gla}}(n+2\mu)$. By specializing their more general damping-mass model to the massless case, the paper \cite{Palmieri-Tu=2021} subsequently enlarged this range for a part of the parameter regime (see the description later). The full shifted Glassey range was extended in \cite{Hamouda-Hamza-Improvement=2021} to $1<p\leqslant p_{\mathrm{Gla}}(n+\mu)$, together with the corresponding polynomial and exponential lifespan upper bounds. This indicates that the scale-invariant damping shifts the effective dimension from $n$ to $n+\mu$. On the global in-time existence side, the recent paper \cite{Fino-Hamza=2026} proved, in particular, small data global in-time existence in one space dimension for $p>p_{\mathrm{Gla}}(1+\mu)$ when $0<\mu\leqslant2$. Combined with the known blow-up result, this identifies $p_{\mathrm{Gla}}(1+\mu)$ as the critical exponent in that parameter range.

\subsection{Scale-invariant damping and mass}

The addition of a scale-invariant mass leads to the model
\begin{align}\label{Introduction-Scale-Invariant-Problem}
	u_{tt}-\Delta u+\frac{\mu}{1+t}u_t+\frac{\nu^2}{(1+t)^2}u=|u_t|^p,
\end{align}
where $\mu\geqslant0$ and $\nu^2\geqslant0$. A central quantity in the analysis of the associated linear equation is the discriminant
\begin{align*}
	\delta:=(\mu-1)^2-4\nu^2.
\end{align*}
The temporal components in the explicit representation of the linear problem are described by modified Bessel functions whose order depends on $\sqrt{\delta}$ (see, for example,
\cite{Wirth=2004,Palmieri-Integral=2021}). Accordingly, the regime $\delta\geqslant0$ is associated with a real Bessel order and will be referred to as the dissipation-dominant regime, whereas $\delta<0$ corresponds to the
mass-dominant regime.

For derivative-type nonlinearities, the authors of \cite{Palmieri-Tu=2021} proved finite-time blow-up under the condition $\delta\geqslant0$. More precisely, they obtained the range
\begin{align}\label{Previous-Palmieri-Tu-Range}
	1<p\leqslant p_{\mathrm{Gla}}(n+\sigma)\ \ \mbox{with}\ \ 
	\sigma:=
	\begin{cases}
		\mu+1-\sqrt{\delta}&\text{if}\ \  0\leqslant\delta<1,\\
		\mu&\text{if}\ \  \delta\geqslant1.
	\end{cases}
\end{align}
Their proof relies on an explicit integral representation of the scale-invariant linear equation and the positivity of the corresponding kernels. In the case $0\leqslant\delta<1$, the argument also requires
the vanishing of the initial displacement. The range in \eqref{Previous-Palmieri-Tu-Range} was subsequently improved in \cite{Hamouda-Hamza=2022}. More precisely, for every nonnegative discriminant, the full range $1<p\leqslant p_{\mathrm{Gla}}(n+\mu)$ was obtained, together with the expected subcritical lifespan upper bound.

The mass-dominant regime is more delicate because $\sqrt{\delta}$ is purely imaginary and the positivity mechanism associated with the real-order Bessel representation is no longer directly available. The
strictly subcritical range was recently treated in the manuscript \cite{Hamza=2026}. More precisely, for
$1<p<p_{\mathrm{Gla}}(n+\mu)$, finite-time blow-up was established with the lifespan estimate
\begin{align}\label{Previous-Subcritical-Loss}
	T_\varepsilon \leqslant C_d\,\varepsilon^{-\frac{2(p-1)}{2-(n+d-1)(p-1)}}\  \ \mbox{for}\ \ \mu<d<\frac{p+1}{p-1}-n.
\end{align}
Since $d$ must be strictly larger than $\mu$, the estimate \eqref{Previous-Subcritical-Loss} contains a loss with respect to the
expected lifespan governed by the effective dimension $n+\mu$. The critical exponent in the mass-dominant regime was considered in the subsequent work \cite{Hamza=2026-2}. By introducing a logarithmic correction into a family of positive test functions, finite-time blow-up was obtained at $p=p_{\mathrm{Gla}}(n+\mu)$. The resulting critical lifespan estimate, however, is of the form
\begin{align}\label{Previous-Critical-Loss}
	T_\varepsilon\leqslant \exp\big(C\varepsilon^{-(p-1+d)}\big)\ \ \text{for every}\ \  d>0,
\end{align}
and therefore still contains an arbitrarily small loss due to $d>0$. Moreover, both \cite{Hamza=2026,Hamza=2026-2} are restricted to the exact scale-invariant coefficients in \eqref{Introduction-Scale-Invariant-Problem}. 

Consequently, a general approach that treats time-dependent damping and mass coefficients, includes the entire mass-dominant regime, and yields the subcritical and critical lifespan estimates without loss has remained unavailable.

\subsection{Main contributions}
For general time-dependent damping and mass coefficients, an explicit positive solution to the adjoint equation is not available. Even for the scale-invariant model, the positivity argument based on the Bessel-function representation depends on the sign of the discriminant and does not apply to the mass-dominant regime. To avoid this restriction, we apply a Liouville transformation and reduce the temporal adjoint equation to an equation involving the effective potential $V(t)$ defined in \eqref{Definition-BV}. If $V\in L^1(\Rplusast)$, the corresponding Volterra integral equation posed from infinity admits a positive solution for a sufficiently large spectral parameter. The resulting exact adjoint solution requires neither an explicit representation of the linear propagator nor a sign condition on the scale-invariant discriminant. Using this solution and two weighted functionals, we derive a divergence condition for finite-time blow-up and an inverse-function upper bound for the lifespan. When the primitive $B(t)$ grows at most logarithmically, the full shifted Glassey range, including the critical exponent, follows with the corresponding lifespan estimates and without loss in the effective dimension. This also covers the mass-dominant scale-invariant case and admissible oscillatory or nondecaying perturbations of the coefficients.

\medskip
\paragraph{Notation.}
We write $\Rplus:=[0,\infty)$ and $\Rplusast:=(0,\infty)$. For $R>0$, we denote by $B_R$ the open ball in $\mb{R}^n$ centered at the origin with radius $R$, and by $\overline{B}_R$ its closure. Unless otherwise specified, all Lebesgue and Sobolev spaces are defined over $\mb{R}^n$. In particular, we write $L^q:=L^q(\mb{R}^n)$ and $H^1:=H^1(\mb{R}^n)$. We denote by $\ml{C}_b(\Rplus)$ the space of bounded continuous functions on $\Rplus$, and by $\supp f$ the support of a function $f$. For $p>1$, we write $p':=\frac{p}{p-1}$ for its H\"older conjugate. For nonnegative quantities $f$ and $g$, the notation $f\lesssim g$ means that $f\leqslant Cg$ for some constant $C>0$, while $f\approx g$ means that both $f\lesssim g$ and $g\lesssim f$ hold. Subscripts attached to $\lesssim$ indicate the permitted dependence of the implicit constant. The symbol $C>0$ denotes a generic constant that may change from line to line and may depend on fixed parameters and the initial data, but is independent of $t$ and $\varepsilon$.

\section{Main results}\label{Section-Main}
\subsection{Energy solutions and blow-up results}

We formulate the solution concept in the natural energy space. The support condition included below is the finite propagation property required in the test function argument.
\begin{defn}[Local in-time energy solution]\label{Def-Energy-Solution}
	Let $T>0$. Let $(u_0,u_1)\in H^1\times L^2$, and assume that $\operatorname{supp}u_0\cup\operatorname{supp}u_1\subset B_R$ for some $R>0$. A function
	\begin{align*}
		u\in \mathcal{C}\big([0,T),H^1\big)\cap \mathcal{C}^1\big([0,T),L^2\big)\ \ \text{satisfying}\ \ u_t\in L^p_{\mathrm{loc}}\big([0,T)\times\mb{R}^n\big),
	\end{align*}
	is called a local in-time energy solution to \eqref{Main-Problem} on $[0,T)$ if
	\begin{align}\label{Support-Condition}
		\operatorname{supp}u(t,\cdot)\subset B_{R+t}\ \ \text{for every}\ \ t\in[0,T),
	\end{align}
	the initial conditions in \eqref{Main-Problem} hold in $H^1\times L^2$, and
	\begin{align}\label{Weak-Formulation}
		&\int_{\mb{R}^n}u_t(t,x)\Phi(t,x)\,\mathrm{d}x-\varepsilon\int_{\mb{R}^n}u_1(x)\Phi(0,x)\,\mathrm{d}x-\int_0^t\int_{\mb{R}^n}u_t(s,x)\Phi_s(s,x)\,\mathrm{d}x\,\mathrm{d}s\notag\\
		&\quad+\int_0^t\int_{\mb{R}^n}\nabla u(s,x)\cdot\nabla\Phi(s,x)\,\mathrm{d}x\,\mathrm{d}s+\int_0^t\int_{\mb{R}^n}\big(b(s)u_t(s,x)+m(s)u(s,x)\big)\Phi(s,x)\,\mathrm{d}x\,\mathrm{d}s\notag\\
		&=\int_0^t\int_{\mb{R}^n}|u_t(s,x)|^p\Phi(s,x)\,\mathrm{d}x\,\mathrm{d}s
 	\end{align}
	for every $t\in[0,T)$ and every $\Phi\in\mathcal{C}_0^\infty\big([0,T)\times\mb{R}^n\big)$.
\end{defn}

\begin{remark}\label{Rem-Support-Time-Derivative}
	The support condition \eqref{Support-Condition} also implies
	\begin{align}\label{Support-Time-Derivative}
		\operatorname{supp}u_t(t,\cdot) \subset \overline{B}_{R+t}\ \ \text{for every}\ \ t\in[0,T).
	\end{align}
    Indeed, this follows by testing $u(s,\cdot)$ against any function compactly supported outside $\overline{B}_{R+t}$ and differentiating with respect to $s$ at $s=t$.
\end{remark}
For the local in-time energy solution under consideration, we denote by $T_\varepsilon\in(0,\infty]$ its maximal continuation time within the class of Definition~\ref{Def-Energy-Solution}, and call $T_\varepsilon$ its lifespan.

Integrating by parts in time in the terms involving $u_t$ and in space in the term involving $\nabla u$, we obtain the following equivalent formulation associated with the formal adjoint operator.
\begin{lemma}[Adjoint weak formulation]\label{Lemma-Adjoint-Weak}
	Let $u$ be an energy solution in the sense of Definition~\ref{Def-Energy-Solution}. The following identity holds:
	\begin{align*}
		&\int_{\mb{R}^n}\big[u_t(t,x)\Phi(t,x)-u(t,x)\Phi_t(t,x)+b(t)u(t,x)\Phi(t,x)\big]\,\mathrm{d}x\\
		&\quad+\int_0^t\int_{\mb{R}^n}u(s,x)\big[\Phi_{ss}(s,x)-\Delta\Phi(s,x)-\partial_s\big(b(s)\Phi(s,x)\big)+m(s)\Phi(s,x)\big]\,\mathrm{d}x\,\mathrm{d}s\\
		&=\int_0^t\int_{\mb{R}^n}|u_t(s,x)|^p\Phi(s,x)\,\mathrm{d}x\,\mathrm{d}s+\varepsilon\int_{\mb{R}^n}\big[\big(u_1(x)+b(0)u_0(x)\big)\Phi(0,x)-u_0(x)\Phi_t(0,x)\big]\,\mathrm{d}x
	\end{align*}
	for every $t\in[0,T)$ and every $\Phi\in\mathcal{C}_0^\infty\big([0,T)\times \mb{R}^n\big)$.
\end{lemma}

The general blow-up criterion and the corresponding lifespan estimate are formulated in terms of the following time integral.
\begin{defn}[Glassey-type time integral]\label{Def-Glassey-Time-Integral}
	For $p>1$, the function $\mathcal{H}_p:[1,\infty)\to\Rplus$ defined via
	\begin{align*}
		\mathcal{H}_p(t):=\int_1^t\mathrm{e}^{-\frac{p-1}{2}B(s)}(1+s)^{-\frac{p-1}{p_{\mathrm{Gla}}(n)-1}}\,\mathrm{d}s \ \ \text{for}\ \  t\geqslant1,
	\end{align*}
	is called the Glassey-type time integral associated with $B$. Here, the exponent of $(1+s)$ is understood as $0$ when $n=1$.
\end{defn}
Since the integrand in the definition of $\mathcal{H}_p$ is continuous and strictly positive, the function $\mathcal{H}_p$ is continuous and strictly increasing on $[1,\infty)$, with $\mathcal{H}_p(1)=0$. Especially, provided that
\begin{align*}
	\lim_{t\to\infty}\mathcal{H}_p(t)=\infty,
\end{align*}
then $\mathcal{H}_p$ is a bijection from $[1,\infty)$ onto $\Rplus$, and its inverse $\mathcal{H}_p^{-1}:\Rplus\to[1,\infty)$ is well defined, continuous, and strictly increasing.

With these preparations in place, we now state the main result of this manuscript.
\begin{theorem}[General blow-up criterion and lifespan estimate]\label{Thm-General-Criterion}
	Let $n\geqslant1$ and $p>1$, and let Assumption~\ref{Assumption-Coefficients} be satisfied. Suppose that $(u_0,u_1)\in H^1\times L^2$ are nonnegative and compactly supported, not both identically zero, with $\supp u_0\cup\supp u_1\subset B_R$ for some $R>0$. For $\varepsilon>0$, let $u$ be a local in-time energy solution to \eqref{Main-Problem} in the sense of Definition~\ref{Def-Energy-Solution}, with lifespan $T_\varepsilon$. If
\begin{align}\label{Divergence-Condition}
	\lim_{t\to\infty}\mathcal{H}_p(t)=\infty,
\end{align}
	then $T_\varepsilon<\infty$. More precisely,
	\begin{align}\label{General-Lifespan-Bound}
		T_\varepsilon\leqslant\mathcal{H}_p^{-1}\big(C\varepsilon^{-(p-1)}\big),
	\end{align}
	where $C=C(n,p,b,m,R,u_0,u_1)>0$ is independent of $\varepsilon$.
\end{theorem}

The next result derives an explicit Glassey-type range from a logarithmic upper bound for $B$. No lower bound for $B$ is required.

\begin{coro}[Glassey-type range under logarithmic growth of $B$]\label{Coro-Log-Damping}
	Let $n\geqslant1$ and $p>1$, and let Assumption~\ref{Assumption-Coefficients} be satisfied. Suppose that $(u_0,u_1)\in H^1\times L^2$ satisfy the sign, support, and nontriviality assumptions of Theorem~\ref{Thm-General-Criterion}. For $\varepsilon>0$, let $u$ be a local in-time energy solution to \eqref{Main-Problem} in the sense of Definition~\ref{Def-Energy-Solution}, with lifespan $T_\varepsilon$. Assume, in addition, that there exist constants $\mu\geqslant0$ and $C_B\geqslant0$ such that
	\begin{align}\label{Log-Damping-Assumption}
		B(t)\leqslant\mu\log(1+t)+C_B\ \ \text{for every}\ \ t\geqslant0.
	\end{align}
	\begin{enumerate}[label=(\roman*)]
		\item If $n+\mu>1$ and $1<p\leqslant p_{\mathrm{Gla}}(n+\mu)$, then $T_\varepsilon<\infty$ for every $\varepsilon>0$. Moreover, for every sufficiently small $\varepsilon>0$, one has
		\begin{align}\label{Explicit-Lifespan-Bounds}
			T_\varepsilon\leqslant
			\begin{cases}
				C\varepsilon^{-\frac{2(p-1)}{2-(n+\mu-1)(p-1)}}&\text{if}\ \ 1<p<p_{\mathrm{Gla}}(n+\mu),\\
				\exp\big(C\varepsilon^{-(p-1)}\big)&\text{if}\ \ p=p_{\mathrm{Gla}}(n+\mu).
			\end{cases}
		\end{align}
		\item If $n=1$ and $\mu=0$, then $T_\varepsilon<\infty$ for every $p>1$  and every $\varepsilon>0$. Moreover, for every sufficiently small $\varepsilon>0$, one has
		\begin{align}\label{One-Dimensional-Zero-Shift}
			T_\varepsilon\leqslant C\varepsilon^{-(p-1)}.
		\end{align}
	\end{enumerate}
	Here, the positive constant $C$ is independent of $\varepsilon$.
\end{coro}

\begin{remark}\label{Rem-Scattering-Case}
	Consider the massless case $m\equiv0$. If $b\geqslant0$ as well as $b\in L^1(\Rplusast)$, then $B(t)\leqslant\|b\|_{L^1(\Rplusast)}$ for every $t\geqslant0$. Hence, \eqref{Log-Damping-Assumption} is satisfied with $\mu=0$. If, in addition, $b$ satisfies the regularity requirements in Assumption~\ref{Assumption-Coefficients} and
	\begin{align*}
		V(t)\big|_{m(t)\equiv 0}=-\frac12b'(t)-\frac14[b(t)]^2\in L^1(\Rplus),
	\end{align*}
	then Corollary~\ref{Coro-Log-Damping} yields the classical Glassey range $1<p\leqslant p_{\mathrm{Gla}}(n)$ for $n\geqslant 2$, while every $p>1$ is covered when $n=1$. Moreover, the lifespan estimates \eqref{Explicit-Lifespan-Bounds} coincide with those obtained in \cite[Theorem~2.1 and Remark~2.2]{Lai-Takamura=2019} for nonnegative integrable damping coefficients.
	
	The coefficient assumptions in the two results are not directly comparable. The result in \cite{Lai-Takamura=2019} is formulated under the conditions $b\geqslant0$ and $b\in L^1(\Rplusast)$, whereas our argument requires the integrability of $-\frac12b'-\frac14b^2$ in the massless case and only an upper bound for $B$. The advantage of the present formulation becomes more apparent when $m$ is nonzero. Since the structural condition is imposed on the combined quantity
	\begin{align*}
		V(t)=m(t)-\frac12b'(t)-\frac14[b(t)]^2,
	\end{align*}
	the mass coefficient itself is not required to be integrable or to have a fixed sign, and cancellations between the damping and mass coefficients are allowed.
\end{remark}

Scale-invariant wave equations with a derivative-type nonlinearity, both in the massless case and in the presence of a scale-invariant mass, have been studied in \cite{Lai-Takamura=2019,Palmieri-Tu=2021,Fino-Hamza=2026,Hamza=2026,Hamza=2026-2} and references given therein. The following corollary specializes both the general nonexistence criterion and the lifespan estimates to the scale-invariant damping-mass model.

\begin{coro}[Scale-invariant damping and mass]
	\label{Coro-Scale-Invariant}
	Let $\mu\geqslant0$ and $\nu^2\geqslant0$, and define
	\begin{align}\label{Discriminant}
		\mb{R}\ni\delta:=(\mu-1)^2-4\nu^2.
	\end{align}
	Let us consider the following Cauchy problem:
	\begin{align}\label{Scale-Invariant-Problem}
		\begin{cases}
			u_{tt}-\Delta u+\dfrac{\mu}{1+t}u_t+\dfrac{\nu^2}{(1+t)^2}u=|u_t|^p,&x\in\mb{R}^n,\ t\in\Rplusast,\\
			u(0,x)=\varepsilon u_0(x),\ \ u_t(0,x)=\varepsilon u_1(x),&x\in\mb{R}^n.
		\end{cases}
	\end{align}
	Suppose that $(u_0,u_1)\in H^1\times L^2$ are nonnegative and compactly supported, not both identically zero, with $\supp u_0\cup\supp u_1\subset B_R$ for some $R>0$. For $\varepsilon>0$, let $u$ be a local in-time energy solution to \eqref{Scale-Invariant-Problem} in the corresponding sense of Definition~\ref{Def-Energy-Solution}, with lifespan $T_\varepsilon$.
	\begin{enumerate}[label=(\roman*)]
	\item If $n+\mu>1$ and $1<p\leqslant p_{\mathrm{Gla}}(n+\mu)$, then $T_\varepsilon<\infty$ for every $\varepsilon>0$. Moreover, for every sufficiently small $\varepsilon>0$, one has
	\begin{align}\label{Scale-Invariant-Lifespan-Bounds}
		T_\varepsilon\leqslant
		\begin{cases}
			C\varepsilon^{-\frac{2(p-1)}{2-(n+\mu-1)(p-1)}}&\text{if}\ \ 1<p<p_{\mathrm{Gla}}(n+\mu),\\
			\exp\big(C\varepsilon^{-(p-1)}\big)&\text{if}\ \ p=p_{\mathrm{Gla}}(n+\mu).
		\end{cases}
	\end{align}
	\item If $n=1$ and $\mu=0$, then $T_\varepsilon<\infty$ for every $p>1$ and every $\varepsilon>0$. Moreover, for every sufficiently small $\varepsilon>0$, one has
	\begin{align*}
		T_\varepsilon\leqslant C\varepsilon^{-(p-1)}.
	\end{align*}
	\end{enumerate}
	Here, the positive constant $C$ is independent of $\varepsilon$.
\end{coro}

\begin{remark}\label{Rem-Role-Discriminant}
	Concerning the coefficients in \eqref{Scale-Invariant-Problem}, we compute
	\begin{align*}
		B(t)=\mu\log(1+t)\ \ \text{and}\ \ V(t)=\frac{1-\delta}{4(1+t)^2}.
	\end{align*}
	Consequently, $V\in L^1(\Rplusast)$ for every $\delta\in\mb{R}$, while \eqref{Log-Damping-Assumption} holds with logarithmic coefficient $\mu$. Namely, the sign of $\delta$ does not affect the applicability of Corollary~\ref{Coro-Log-Damping}.
\end{remark}

\begin{remark}\label{Rem-Scale-Invariant-Comparison}
	For $\delta\geqslant1$, the blow-up range and the lifespan upper bounds in Corollary~\ref{Coro-Scale-Invariant} coincide with those obtained in \cite{Palmieri-Tu=2021}. When $0\leqslant\delta<1$, the result in \cite{Palmieri-Tu=2021} involves a $\delta$-dependent shift of the Glassey exponent and an additional restriction on the initial displacement, whereas Corollary~\ref{Coro-Scale-Invariant} yields the full range $1<p\leqslant p_{\mathrm{Gla}}(n+\mu)$ for general nonnegative initial data. The restriction on the sign of $\delta$ was removed in \cite{Hamza=2026} for the subcritical case, and in \cite{Hamza=2026-2} for the critical case. The present result also provides the corresponding polynomial and exponential lifespan upper bounds.
\end{remark}

\begin{remark}\label{Rem-Short-Range-Perturbations}
	The conclusions of Corollary~\ref{Coro-Scale-Invariant} remain valid under short-range perturbations of the scale-invariant coefficients. More precisely, let us consider
	\begin{align*}
		b(t)=\frac{\mu}{1+t}+r_b(t)\ \ \text{and}\ \ m(t)=\frac{\nu^2}{(1+t)^2}+r_m(t),
	\end{align*}
    with $\mu\geqslant0$ and $\nu^2\geqslant0$, where $r_b\in\ml{C}^1(\Rplus)\cap W^{1,1}(\Rplusast)$ and $r_m\in\ml{C}(\Rplus)\cap L^1(\Rplusast)$. In this case, we examine
	\begin{align*}
		B(t)=\mu\log(1+t)+\int_0^t r_b(s)\,\mathrm{d}s\leqslant\mu\log(1+t)+\|r_b\|_{L^1(\Rplusast)}
	\end{align*}
	and
	\begin{align*}
		V(t)=\frac{\nu^2+\frac{\mu}{2}-\frac{\mu^2}{4}}{(1+t)^2}+r_m(t)-\frac12r'_b(t)-\frac{\mu}{2(1+t)}r_b(t)-\frac14[r_b(t)]^2\in L^1(\Rplusast),
	\end{align*}
	in which $r_b\in L^1(\Rplusast)\cap L^\infty(\Rplusast)\subset L^2(\Rplusast)$ from $r_b\in W^{1,1}(\Rplusast)$ was used.
\end{remark}

\subsection{Coefficient assumptions and representative examples}\label{Subsection-Coefficient-Examples}

Assumption~\ref{Assumption-Coefficients} is formulated in terms of the quantities $B$ and $V$ defined in \eqref{Definition-BV}. The regularity conditions provide a natural class in which the equation and these quantities are well defined. Its main structural requirement is
\begin{align*}
	V\in L^1(\Rplusast),
\end{align*}
which is imposed on the combined contribution of the damping and mass coefficients, rather than on $b$ and $m$ separately. In particular, neither coefficient is required to have a fixed sign, and cancellations between their non-integrable components are allowed.

The upper bound for $B$ in \eqref{Log-Damping-Assumption} is a separate condition that converts the general criterion into the explicit Glassey-type range and lifespan estimates of Corollary~\ref{Coro-Log-Damping}. The following examples illustrate the coefficient classes covered by the main results.

\begin{exam}[Algebraically decaying scattering coefficients]\label{Example-Scattering}
	Let us consider
	\begin{align*}
		b(t)=\frac{b_0}{(1+t)^\beta} \ \ \text{and}\ \ m(t)=\frac{m_0}{(1+t)^\gamma},
	\end{align*}
	where $b_0,m_0\in\mb{R}$ and $\beta,\gamma>1$. Thanks to $b'(t)=-\frac{\beta b_0}{(1+t)^{\beta+1}}$, we have
	\begin{align*}
		V(t)=\frac{m_0}{(1+t)^\gamma}+\frac{\beta b_0}{2(1+t)^{\beta+1}}-\frac{b_0^2}{4(1+t)^{2\beta}}\in L^1(\Rplusast).
	\end{align*}
	Moreover, $B$ is bounded on $\Rplus$. Therefore, Corollary~\ref{Coro-Log-Damping} applies with $\mu=0$ and yields the classical Glassey range and the corresponding lifespan upper bounds.
\end{exam}

\begin{exam}[Scale-invariant damping and mass]\label{Example-Scale-Invariant}
	Let us consider
	\begin{align*}
		b(t)=\frac{\mu}{1+t}\ \ \text{and}\ \ m(t)=\frac{\nu^2}{(1+t)^2},
	\end{align*}
	with $\mu\geqslant0$ and $\nu^2\geqslant0$. This is precisely the model considered in Corollary~\ref{Coro-Scale-Invariant}. In particular, the stated Glassey-type blow-up range and lifespan estimates hold independently of the sign of $\delta$ defined in \eqref{Discriminant}.
\end{exam}

\begin{exam}[Algebraically decaying perturbations]\label{Example-Short-Range-Perturbations}
	Let us consider
	\begin{align*}
		b(t)=\frac{\mu}{1+t}+\frac{b_1}{(1+t)^\beta}\ \ \text{and}\ \ m(t)=\frac{\nu^2}{(1+t)^2}+\frac{m_1}{(1+t)^\gamma},
	\end{align*}
	with $\mu\geqslant0$ and $\nu^2\geqslant0$, where $b_1,m_1\in\mb{R}$ and $\beta,\gamma>1$.
	The perturbations
	\begin{align*}
		r_b(t)=b_1(1+t)^{-\beta} \ \ \text{and}\ \ r_m(t)=m_1(1+t)^{-\gamma}
	\end{align*}
	satisfy the assumptions of Remark~\ref{Rem-Short-Range-Perturbations}. Hence, the Glassey-type blow-up range and the lifespan upper bounds are the same as those for the unperturbed scale-invariant model.
\end{exam}

\begin{exam}[Non-integrable oscillatory perturbations]\label{Example-Oscillatory}
	Fixing $0<\alpha\leqslant1$ and $b_2\in\mb{R}\setminus\{0\}$, motivated by \cite{Ghisi-Gobbino=2024,Ghisi-Gobbino=2025}, we introduce the
	non-integrable decaying oscillatory perturbation
	\begin{align*}
		r_\alpha(t):=\frac{b_2\sin t}{(1+t)^{\alpha}}.
	\end{align*}
	Let us consider
	\begin{align*}
		b(t)=\frac{\mu}{1+t}+r_\alpha(t)\ \ \text{and}\ \ m(t)=\frac{\nu^2}{(1+t)^2}+\frac12r_\alpha'(t)+\frac{\mu}{2(1+t)}r_\alpha(t)+\frac14[r_\alpha(t)]^2+\mathcal{R}(t),
	\end{align*}
	with $\mu\geqslant0$ and $\nu^2\geqslant0$, where $\mathcal{R}\in\ml{C}(\Rplus)\cap L^1(\Rplusast)$. Although $r_\alpha\notin L^1(\Rplusast)$, Dirichlet's test suggests
	\begin{align*}
		\int_0^t r_\alpha(s)\,\mathrm{d}s=O(1)\ \ \text{as}\ \ t\to\infty.
	\end{align*}
	Consequently,
	\begin{align*}
		B(t)=\mu\log(1+t)+O(1).
	\end{align*}
	A direct substitution into the definition of $V$ gives
	\begin{align*}
		V(t)=\frac{1-\delta}{4(1+t)^2}+\mathcal{R}(t)\in L^1(\Rplusast)\ \ \text{with}\ \ \delta=(\mu-1)^2-4\nu^2.
	\end{align*}
	Therefore, Corollary~\ref{Coro-Log-Damping} applies with the same shifted dimension $n+\mu$ as in the unperturbed scale-invariant model.
\end{exam}

\begin{exam}[Nondecaying periodic coefficients]\label{Example-Periodic-Coefficients}
	Let $b_3\in\mb{R}\setminus\{0\}$. Motivated by \cite{Wirth=2008,Girardi-Wirth=2021}, let us consider
	\begin{align*}
		b(t)=b_3\sin t \ \ \text{and}\ \ m(t)=\frac{b_3}{2}\cos t+\frac{b_3^2}{4}\sin^2t.
	\end{align*}
	Although neither coefficient decays as $t\to\infty$, we have
	\begin{align*}
		B(t)=b_3(1-\cos t)\ \ \text{and}\ \ V(t)=0.
	\end{align*}
	Hence, $B$ is bounded and Assumption~\ref{Assumption-Coefficients} is satisfied. Corollary~\ref{Coro-Log-Damping} therefore yields the classical Glassey range and the corresponding lifespan upper bounds.
\end{exam}

\begin{exam}[A logarithmic correction to the scale-invariant damping]\label{Example-Logarithmic-Correction}
	Let $\mu\geqslant0$ and $\nu^2\geqslant0$. Taking
	$\kappa\in\mb{R}$, motivated by \cite{Nascimento-Wirth=2015}, let us consider
	\begin{align*}
		b(t)=\frac{\mu}{1+t}+\frac{\kappa}{(\mathrm{e}+t)\log(\mathrm{e}+t)}\ \ \text{and}\ \  m(t)=\frac{\nu^2}{(1+t)^2}.
	\end{align*}
	A direct computation shows that $V\in L^1(\Rplusast)$, whereas
	\begin{align*}
		B(t)=\mu\log(1+t)+\kappa\log\log(\mathrm{e}+t).
	\end{align*}
	Consequently, the Glassey-type time integral is given by
	\begin{align}\label{Log-Corrected-Time-Integral}
		\ml{H}_p(t)=\int_1^t(1+s)^{-\frac{(n+\mu-1)(p-1)}{2}}[\log(\mathrm{e}+s)]^{-\frac{\kappa(p-1)}{2}}\dd s.
	\end{align}
	\begin{itemize}
		\item In the subcritical case $1<p<p_{\mathrm{Gla}}(n+\mu)$, then the power of $(1+s)$ in \eqref{Log-Corrected-Time-Integral} is strictly smaller than one. Therefore,
		\begin{align*}
			\ml{H}_p(t)\to\infty\ \ \text{as}\ \ t\to\infty,
		\end{align*}
		independently of $\kappa$. Hence, the logarithmic correction does not change the subcritical blow-up range.
		\item In the critical case $p=p_{\mathrm{Gla}}(n+\mu)$ with $n+\mu>1$, i.e. $p-1=\frac{2}{n+\mu-1}$, \eqref{Log-Corrected-Time-Integral} becomes
		\begin{align}\label{Log-Corrected-Critical-Integral}
			\ml{H}_p(t)=\int_1^t\frac{\dd s}{(1+s)[\log(\mathrm{e}+s)]^{\frac{\kappa}{n+\mu-1}}}.
		\end{align}
		It follows that $\ml{H}_p(t)\to\infty$ if and only if $\kappa\leqslant n+\mu-1$.
		\begin{itemize}
			\item If $\kappa<n+\mu-1$, then $\ml{H}_p(t)\approx[\log(\mathrm{e}+t)]^{1-\frac{\kappa}{n+\mu-1}}$ for sufficiently large $t$. Therefore, Theorem~\ref{Thm-General-Criterion} yields
			\begin{align*}
				T_\varepsilon\leqslant\exp\left(C\varepsilon^{-\frac{2}{n+\mu-1-\kappa}}\right)
			\end{align*}
			for every sufficiently small $\varepsilon>0$.
			\item At the borderline value $\kappa=n+\mu-1$, we have $\ml{H}_p(t)\approx\log\log(\mathrm{e}+t)$ for sufficiently large $t$. Hence, Theorem~\ref{Thm-General-Criterion} gives the double-exponential upper bound
			\begin{align*}
				T_\varepsilon\leqslant\exp\left[\exp\left(C\varepsilon^{-\frac{2}{n+\mu-1}}\right)\right]
			\end{align*}
			for every sufficiently small $\varepsilon>0$.
			\item If $\kappa>n+\mu-1$, then the integral in \eqref{Log-Corrected-Critical-Integral} converges as $t\to\infty$. Thus, the critical endpoint is not covered by Theorem~\ref{Thm-General-Criterion} in this regime. 
		\end{itemize}
	\end{itemize}
\end{exam}

\section{Construction of a positive solution to the adjoint equation}\label{Section-Adjoint}

\subsection{The spatial test function}

For $\eta>0$, following \cite[Section 2]{Yordanov-Zhang=2006}, let us define
\begin{align}\label{Spatial-Test-Function}
	\varphi^\eta(x):=
	\begin{cases}
		\mathrm{e}^{\eta x}+\mathrm{e}^{-\eta x}&\text{if}\ \ n=1,\\
		\displaystyle\int_{\mb{S}^{n-1}}\mathrm{e}^{\eta x\cdot\omega}\dd\sigma_\omega&\text{if}\ \ n\geqslant2,
	\end{cases}
\end{align}
where $\mb{S}^{n-1}$ denotes the unit sphere in $\mb{R}^n$, endowed with the standard surface measure $\dd\sigma_\omega$. It satisfies the following modified Helmholtz equation:
\begin{align}\label{Spatial-Eigenfunction}
	\Delta\varphi^\eta=\eta^2\varphi^\eta\ \ \text{in}\ \mb{R}^n
\end{align}
and $\varphi^\eta(x)>0$. For $n\geqslant2$, by scaling the standard estimate in \cite[Lemma~2.1]{Yordanov-Zhang=2006}, we obtain
\begin{align}\label{Spatial-Pointwise-Estimate}
	\varphi^\eta(x)\lesssim\mathrm{e}^{\eta|x|}(1+\eta|x|)^{-\frac{n-1}{2}}\lesssim_{\eta}\mathrm{e}^{\eta|x|}(1+|x|)^{-\frac{n-1}{2}}.
\end{align}
For $n=1$, the corresponding estimate follows directly from \eqref{Spatial-Test-Function}.

\subsection{Reduction of the temporal adjoint equation}

The formal adjoint operator associated with the linear part of \eqref{Main-Problem} is
\begin{align*}
	\ml{L}_{b,m}^*\Psi:=\Psi_{tt}-\Delta\Psi-\partial_t\big(b(t)\Psi\big)+m(t)\Psi.
\end{align*}
We seek a separated solution of the form
\begin{align*}
	\Psi^\eta(t,x)=\lambda^\eta(t)\varphi^\eta(x).
\end{align*}
By \eqref{Spatial-Eigenfunction}, the condition $\ml{L}_{b,m}^*\Psi^\eta=0$ is equivalent to
\begin{align}\label{Temporal-Adjoint-Equation}
	(\lambda^\eta)''-(b\lambda^\eta)'+(m-\eta^2)\lambda^\eta=0.
\end{align}
To reduce \eqref{Temporal-Adjoint-Equation} to an equation involving the structural potential $V$, we employ the following Liouville transformation.

\begin{lemma}[Liouville reduction]\label{Lemma-Liouville}
	Let $\eta>0$ and $a\in\ml{C}^2(\Rplus)$. Define
	\begin{align}\label{Liouville-Ansatz}
		\lambda(t):=\mathrm{e}^{-\eta t+\frac{1}{2}B(t)}a(t).
	\end{align}
	Then
	\begin{align}\label{Liouville-Identity}
		\lambda''(t)-\big(b(t)\lambda(t)\big)'+\big(m(t)-\eta^2\big)\lambda(t)=\mathrm{e}^{-\eta t+\frac{1}{2}B(t)}\big[a''(t)-2\eta a'(t)+V(t)a(t)\big].
	\end{align}
\end{lemma}

\begin{proof}
	Since $B'(t)=b(t)$, direct differentiation of \eqref{Liouville-Ansatz} gives
	\begin{align*}
		\lambda'&=\mathrm{e}^{-\eta t+\frac{1}{2}B(t)}\left[a'+\left(-\eta+\frac{1}{2}b\right)a\right],\\
		\lambda''&=\mathrm{e}^{-\eta t+\frac{1}{2}B(t)}\left[a''+2\left(-\eta+\frac{1}{2}b\right)a'+\left\{\left(-\eta+\frac{1}{2}b\right)^2+\frac{1}{2}b'\right\}a\right],
	\end{align*}
	and
	\begin{align*}
		(b\lambda)'=\mathrm{e}^{-\eta t+\frac{1}{2}B(t)}\left[ba'+\left\{b'+b\left(-\eta+\frac{1}{2}b\right)\right\}a\right].
	\end{align*}
	Consequently,
	\begin{align*}
		\lambda''-(b\lambda)'+(m-\eta^2)\lambda&=\mathrm{e}^{-\eta t+\frac{1}{2}B(t)}\left[a''-2\eta a'+\left(m-\frac{1}{2}b'-\frac{1}{4}b^2\right)a\right]\\
		&=\mathrm{e}^{-\eta t+\frac{1}{2}B(t)}\big(a''-2\eta a'+Va\big),
	\end{align*}
	which proves \eqref{Liouville-Identity}.
\end{proof}

It follows from \cref{Lemma-Liouville} that it remains to construct a positive solution to
\begin{align}\label{Amplitude-Equation}
	a_\eta''(t)-2\eta a_\eta'(t)+V(t)a_\eta(t)=0
\end{align}
satisfying
\begin{align}\label{Amplitude-Asymptotic-Conditions}
	\lim_{t\to\infty}a_\eta(t)=1\ \ \text{as well as}\ \ \lim_{t\to\infty}a_\eta'(t)=0.
\end{align}

\subsection{The Volterra equation}

We first derive the integral equation associated with the asymptotic conditions in \eqref{Amplitude-Asymptotic-Conditions}. Suppose temporarily that $a$ is a bounded solution to \eqref{Amplitude-Equation} satisfying these conditions. Multiplying \eqref{Amplitude-Equation} by $\mathrm{e}^{-2\eta t}$ gives
\begin{align*}
	\frac{\dd}{\dd t}\left(\mathrm{e}^{-2\eta t}a'(t)\right)=-\mathrm{e}^{-2\eta t}V(t)a(t).
\end{align*}
Integrating over $[t,\infty)$, we obtain
\begin{align*}
	a'(t)=\int_t^\infty\mathrm{e}^{-2\eta(s-t)}V(s)a(s)\dd s.
\end{align*}
A further integration and Fubini's theorem yield
\begin{align*}
	a(t)&=1-\int_t^\infty\int_\tau^\infty\mathrm{e}^{-2\eta(s-\tau)}V(s)a(s)\dd s\dd\tau\\
	&=1-\frac{1}{2\eta}\int_t^\infty\big(1-\mathrm{e}^{-2\eta(s-t)}\big)V(s)a(s)\dd s.
\end{align*}
Here, the use of Fubini's theorem is justified by $V\in L^1(\Rplusast)$ and the boundedness of $a$. Therefore, every bounded solution of the asymptotic problem satisfies the Volterra equation introduced below.

Let us introduce $M_V:=\|V\|_{L^1(\Rplusast)}$. Then we have the next preliminary.

\begin{prop}[Positive asymptotic amplitude]\label{Prop-Volterra}
	Let $\eta\geqslant\max\{1,16M_V\}$. Then there exists a unique $a_\eta\in\ml{C}_b(\Rplus)$ satisfying
	\begin{align}\label{Volterra-Equation}
		a_\eta(t)=1-\frac{1}{2\eta}\int_t^\infty\big(1-\mathrm{e}^{-2\eta(s-t)}\big)V(s)a_\eta(s)\dd s.
	\end{align}
	Moreover, $a_\eta\in\ml{C}^2(\Rplus)$ fulfills
	\begin{align}\label{Amplitude-Bounds}
		\frac{30}{31}\leqslant a_\eta(t)\leqslant\frac{32}{31}\ \ \text{for every}\ \ t\geqslant0,
	\end{align}
	and $a_\eta$ solves \eqref{Amplitude-Equation} and satisfies \eqref{Amplitude-Asymptotic-Conditions}.
\end{prop}

\begin{proof}
Let us consider the Banach space
\begin{align*}
	X:=\ml{C}_b(\Rplus) \ \ \text{endowed with}\ \ \|a\|_X:=\sup_{t\geqslant0}|a(t)|.
\end{align*}
For $a\in X$, define
\begin{align*}
	(\ml{K}_\eta a)(t):=\frac{1}{2\eta}\int_t^\infty\big(1-\mathrm{e}^{-2\eta(s-t)}\big)V(s)a(s)\dd s.
\end{align*}
Since $0\leqslant1-\mathrm{e}^{-2\eta(s-t)}\leqslant1$ for $s\geqslant t$, we have
\begin{align}\label{Volterra-Norm}
	\|\ml{K}_\eta a\|_X\leqslant\frac{M_V}{2\eta}\|a\|_X.
\end{align}
The dominated convergence theorem and \eqref{Volterra-Norm} show that $\ml{K}_\eta$ maps $X$ into itself.

Define the affine map $\ml{T}_\eta a:=1-\ml{K}_\eta a$. By \eqref{Volterra-Norm} and $\eta\geqslant16M_V$, for any $a_1,a_2\in X$, one has
\begin{align*}
	\|\ml{T}_\eta a_1-\ml{T}_\eta a_2\|_X\leqslant\frac{1}{32}\|a_1-a_2\|_X.
\end{align*}
Thus, the Banach fixed point theorem yields a unique $a_\eta\in X$ satisfying \eqref{Volterra-Equation}.

Setting $\kappa_\eta:=\frac{M_V}{2\eta}$, from $\eta\geqslant16M_V$, we have $\kappa_\eta\leqslant\frac{1}{32}$. By \eqref{Volterra-Equation} and \eqref{Volterra-Norm}, one derives
\begin{align}\label{Amplitude-Upper}
	\|a_\eta\|_X\leqslant1+\kappa_\eta\|a_\eta\|_X\leqslant\frac{32}{31}.
\end{align}
Using \eqref{Volterra-Equation} once more, we obtain
\begin{align*}
	\|a_\eta-1\|_X\leqslant\kappa_\eta\|a_\eta\|_X\leqslant\frac{1}{31},
\end{align*}
which proves \eqref{Amplitude-Bounds}.

We next verify the regularity of $a_\eta$. Differentiating \eqref{Volterra-Equation} under the integral sign gives
\begin{align}\label{First-Derivative-Amplitude}
	a_\eta'(t)&=\int_t^\infty\mathrm{e}^{-2\eta(s-t)}V(s)a_\eta(s)\dd s,\\
	a_\eta''(t)&=-V(t)a_\eta(t)+2\eta a_\eta'(t).\notag
\end{align}
Therefore, $a_\eta\in\ml{C}^2(\Rplus)$ and solves \eqref{Amplitude-Equation}.

Eventually, recalling $V\in L^1(\Rplusast)$, it follows from \eqref{Volterra-Equation}, \eqref{Amplitude-Upper} and \eqref{First-Derivative-Amplitude} that
\begin{align*}
	|a_\eta(t)-1|&\leqslant\frac{16}{31\eta}\int_t^\infty|V(s)|\dd s\to0,\\
	|a_\eta'(t)|&\leqslant\frac{32}{31}\int_t^\infty|V(s)|\dd s\to0,
\end{align*}
as $t\to\infty$. This proves \eqref{Amplitude-Asymptotic-Conditions}.
\end{proof}

\begin{remark}
	\label{Remark-First-Correction}
	Let us consider
	\begin{align*}
		V(t)=\frac{q_0}{(1+t)^2}\ \ \text{with}\ \ q_0\in\mb{R}.
	\end{align*}
	For $\eta\geqslant\eta_*:=\max\{1,16|q_0|\}$, \eqref{Volterra-Equation} and \eqref{Amplitude-Upper} imply
	\begin{align}\label{First-Correction-Rough}
		|a_\eta(t)-1|\leqslant\frac{|q_0|}{2\eta}\|a_\eta\|_{L^\infty(\Rplusast)}\int_t^\infty(1+s)^{-2}\dd s \leqslant\frac{16|q_0|}{31\eta(1+t)}
	\end{align}
	for every $t\geqslant0$. We now rewrite \eqref{Volterra-Equation} as
	\begin{align*}
		a_\eta(t)-1&=-\frac{q_0}{2\eta}\int_t^\infty(1+s)^{-2}\dd s+\frac{q_0}{2\eta}\int_t^\infty\mathrm{e}^{-2\eta(s-t)}(1+s)^{-2}\dd s\\
		&\quad -\frac{q_0}{2\eta}\int_t^\infty\left(1-\mathrm{e}^{-2\eta(s-t)}\right)(1+s)^{-2}\big(a_\eta(s)-1\big)\dd s.
	\end{align*}
	The first term on the right-hand side is given by
	\begin{align*}
		-\frac{q_0}{2\eta}\int_t^\infty(1+s)^{-2}\dd s=-\frac{q_0}{2\eta(1+t)}.
	\end{align*}
	For the second term, we have
	\begin{align*}
		\int_t^\infty\mathrm{e}^{-2\eta(s-t)}(1+s)^{-2}\dd s\leqslant\frac{1}{(1+t)^2}\int_t^\infty\mathrm{e}^{-2\eta(s-t)}\dd s=\frac{1}{2\eta(1+t)^2}.
	\end{align*}
	Moreover, it follows from \eqref{First-Correction-Rough} that
	\begin{align*}
	\int_t^\infty\left(1-\mathrm{e}^{-2\eta(s-t)}\right)(1+s)^{-2}|a_\eta(s)-1|\dd s\leqslant\frac{16|q_0|}{31\eta}\int_t^\infty(1+s)^{-3}\dd s=\frac{8|q_0|}{31\eta(1+t)^2}.
	\end{align*}
	Combining the preceding estimates, we obtain
	\begin{align*}
		a_\eta(t)=1-\frac{q_0}{2\eta(1+t)}+O\left(\frac{1}{\eta^2(1+t)^2}\right).
	\end{align*}
	The implicit constant depends only on $q_0$ and is uniform for $t\geqslant0$ and $\eta\geqslant\eta_*$. In particular, for $q_0=\frac{1-\delta}{4}$, with $\delta$ defined in \eqref{Discriminant}, we have
	\begin{align*}
		a_\eta(t)=1-\frac{1-\delta}{8\eta(1+t)}+O\left(\frac{1}{\eta^2(1+t)^2}\right).
	\end{align*}
	Therefore, the first correction associated with the	scale-invariant potential is recovered directly from the Volterra equation.
\end{remark}

\subsection{The positive adjoint solution and its logarithmic derivative}

We use the positive amplitude obtained in Proposition~\ref{Prop-Volterra} to construct an exact solution to the adjoint equation. We normalize its temporal factor at $t=0$ and choose $\eta$ sufficiently large to obtain uniform bounds for its logarithmic derivative.

Let us set $M_b:=\|b\|_{L^\infty(\Rplusast)}$ and fix
\begin{align}\label{Eta-Threshold}
	\eta\geqslant\eta_0:=\max\{2,16M_V,8M_b\}.
\end{align}
We define the temporal profile and the corresponding separated test function by
\begin{align}\label{Temporal-Test-Function}
	\lambda^\eta(t)&:=\mathrm{e}^{-\eta t+\frac{1}{2}B(t)}\frac{a_\eta(t)}{a_\eta(0)},\\\label{Adjoint-Test-Function}
	\Psi^\eta(t,x)&:=\lambda^\eta(t)\varphi^\eta(x).
\end{align}
It follows from Lemma~\ref{Lemma-Liouville} and Proposition~\ref{Prop-Volterra} that
\begin{align}\label{Exact-Adjoint-Equation}
	\ml{L}_{b,m}^*\Psi^\eta=0,\ \ \lambda^\eta(0)=1,\ \ \Psi^\eta(t,x)>0.
\end{align}
Furthermore, \eqref{Amplitude-Bounds} implies
\begin{align}\label{Temporal-Comparability}
	\lambda^\eta(t)\approx\mathrm{e}^{-\eta t+\frac{1}{2}B(t)}\ \ \text{for every}\ \ t\geqslant0.
\end{align}
\begin{remark}\label{Remark-Decaying-Mode}
    For the scale-invariant coefficients $b(t)=\frac{\mu}{1+t}$ and $m(t)=\frac{\nu^2}{(1+t)^2}$, we have
	\begin{align*}
		B(t)=\mu\log(1+t) \ \ \text{and}\ \ V(t)=\frac{1-\delta}{4(1+t)^2}.
	\end{align*}
	Since $a_\eta(t)\to1$ as $t\to\infty$, it follows from \eqref{Temporal-Test-Function} that
	\begin{align*}
		\lambda^\eta(t)\sim\frac{1}{a_\eta(0)}(1+t)^{\frac{\mu}{2}}\mathrm{e}^{-\eta t}\ \ \text{as}\ \ t\to\infty.
	\end{align*}
	This agrees with the large-time behavior of the decaying modified Bessel profile appearing in the explicit representation of the scale-invariant adjoint equation. Thus, the Volterra construction recovers this profile without using special functions and justifies its positivity on $\Rplus$ for sufficiently large $\eta$.
\end{remark}

To prepare the weighted functional estimates, we introduce the logarithmic derivatives
\begin{align*}
	y_\eta(t):=\frac{a_\eta'(t)}{a_\eta(t)} \ \ \text{as well as}\ \ r^\eta(t):=\frac{(\lambda^\eta)'(t)}{\lambda^\eta(t)},
\end{align*}
together with
\begin{align*}
	z_\eta(t):=\eta-y_\eta(t).
\end{align*}
It follows from \eqref{Temporal-Test-Function} that
\begin{align}\label{r-Identity}
	r^\eta(t)=-\eta+\frac{1}{2}b(t)+y_\eta(t)=\frac{1}{2}b(t)-z_\eta(t).
\end{align}
The combinations of $b$ and $r^\eta$ that arise naturally in the weighted adjoint identities are
\begin{align}\label{Beta-Gamma-Definition}
	\beta^\eta(t):=b(t)-r^\eta(t)\ \ \text{and}\ \ \gamma^\eta(t):=b(t)-2r^\eta(t).
\end{align}

\begin{lemma}[Uniform bounds for the adjoint coefficients]\label{Lemma-Coefficient-Bounds}
	For $\eta\geqslant\eta_0$, the following estimates hold for every $t\geqslant0$:
	\begin{align}\label{y-z-Bounds}
		|y_\eta(t)|\leqslant\frac{\eta}{8}\ \ \text{and} \ \ \frac{7\eta}{8}\leqslant z_\eta(t)\leqslant\frac{9\eta}{8},
	\end{align}
moreover,
	\begin{align}\label{Beta-Gamma-Bounds}
		\frac{13\eta}{16}\leqslant\beta^\eta(t)\leqslant\frac{19\eta}{16}\ \ \text{and}\ \ \frac{7\eta}{4}\leqslant\gamma^\eta(t)\leqslant\frac{9\eta}{4}.
	\end{align}
\end{lemma}

\begin{proof}
	By \eqref{First-Derivative-Amplitude} and \eqref{Amplitude-Bounds}, one gets
	\begin{align*}
		|a_\eta'(t)|\leqslant\frac{32}{31}\int_t^\infty|V(s)|\dd s\leqslant\frac{32}{31}M_V.
	\end{align*}
	Since $a_\eta(t)\geqslant\frac{30}{31}$ and $\eta\geqslant16M_V$, we obtain
	\begin{align*}
		|y_\eta(t)|\leqslant\frac{16}{15}M_V\leqslant\frac{\eta}{15}\leqslant\frac{\eta}{8}.
	\end{align*}
	The bounds for $z_\eta$ in \eqref{y-z-Bounds} follow from $z_\eta=\eta-y_\eta$. Moreover, \eqref{r-Identity} and \eqref{Beta-Gamma-Definition} give
	\begin{align*}
		\beta^\eta(t)=z_\eta(t)+\frac{1}{2}b(t)\ \ \text{and}\ \ \gamma^\eta(t)=2z_\eta(t).
	\end{align*}
	Because of $|b(t)|\leqslant M_b\leqslant\frac{\eta}{8}$, \eqref{Beta-Gamma-Bounds} follows from \eqref{y-z-Bounds}.
\end{proof}

\begin{remark}\label{Remark-Weaker-Regularity}
	The fixed point argument for the Volterra equation uses only $V\in L^1(\Rplusast)$. The continuity of $V$ is used to recover the differential equation and the $\ml{C}^2$-regularity of the amplitude. Extending the adjoint weak formulation and the subsequent weighted identities to rougher coefficients would require an approximation argument, including the treatment of the initial trace of $b$ and the passage to the limit in the adjoint equation. We do not pursue this extension and retain Assumption~\ref{Assumption-Coefficients} throughout the paper.
\end{remark}

\begin{remark}\label{Remark-Bounded-Damping}
The Volterra construction of $a_\eta$ uses only the assumption $V\in L^1(\Rplusast)$. The boundedness of $b$ is required in \cref{Lemma-Coefficient-Bounds} to obtain the uniform positivity and comparability of $\beta^\eta$ and $\gamma^\eta$, which are used in the weighted functional estimates. Hence, unbounded damping would require a modification of the comparison argument, while the Volterra construction remains unchanged.
\end{remark}

\section{Lower bound estimates for the time-dependent functionals}\label{Section-Functionals}

\subsection{The adjoint integral identity}

We introduce the weighted functionals associated with $u$, $u_t$, and the nonlinear term:
\begin{align*}
	F^\eta(t)&:=\int_{\mb{R}^n}u(t,x)\Psi^\eta(t,x)\dd x,\\
	G^\eta(t)&:=\int_{\mb{R}^n}u_t(t,x)\Psi^\eta(t,x)\dd x,\\
	N^\eta(t)&:=\int_{\mb{R}^n}|u_t(t,x)|^p\Psi^\eta(t,x)\dd x.
\end{align*}
By \eqref{Support-Condition} and \eqref{Support-Time-Derivative}, these integrals are well defined. Indeed, on every compact interval $[0,T_0]\subset[0,T_\varepsilon)$, the functions $\Psi^\eta$ and
$\Psi_t^\eta$ are bounded on $[0,T_0]\times B_{R+T_0}$. As a consequence,
\begin{align*}
	F^\eta\in\ml{C}^1([0,T_\varepsilon)),\ \ G^\eta\in\ml{C}([0,T_\varepsilon)),\ \ N^\eta\in L^1_{\mathrm{loc}}([0,T_\varepsilon)).
\end{align*}

\begin{lemma}[Weighted adjoint identity]\label{Lemma-Weighted-Identity}
	Let $u$ be a local in-time energy solution to \eqref{Main-Problem} in the sense of Definition~\ref{Def-Energy-Solution}. Then, for every $t\in[0,T_\varepsilon)$,
	\begin{align}\label{Weighted-Identity}
		G^\eta(t)+\beta^\eta(t)F^\eta(t)=\int_0^tN^\eta(s)\dd s+\varepsilon C^\eta(u_0,u_1),
	\end{align}
	where
	\begin{align}\label{Initial-Weighted-Constant}
		C^\eta(u_0,u_1):=\int_{\mb{R}^n}\big[u_1(x)+\beta^\eta(0)u_0(x)\big]\varphi^\eta(x)\dd x.
	\end{align}
	If $u_0$ and $u_1$ are nonnegative and not both identically zero, then
	\begin{align}\label{Positive-Initial-Weighted-Constant}
		C^\eta(u_0,u_1)>0.
	\end{align}
\end{lemma}

\begin{proof}
	By the support condition, $\Psi^\eta$ can be used in the adjoint weak formulation after a standard space-time localization argument. Since $\ml{L}_{b,m}^*\Psi^\eta=0$, $\lambda^\eta(0)=1$, and $\Psi_t^\eta=r^\eta\Psi^\eta$, the adjoint weak formulation gives \eqref{Weighted-Identity}. The positivity of $C^\eta(u_0,u_1)$ follows from $\varphi^\eta>0$, $\beta^\eta(0)>0$, and the assumptions on the initial data.
\end{proof}

\subsection{Estimate for the first weighted functional}

To simplify the weighted adjoint identity, one may define
\begin{align*}
	J^\eta(t):=\int_0^tN^\eta(s)\dd s+\varepsilon C^\eta(u_0,u_1).
\end{align*}
Under the assumptions on the initial data, $J^\eta$ is locally absolutely continuous and satisfies
\begin{align*}
	J^\eta(t)>0 \ \ \text{and}\ \	(J^\eta)'(t)=N^\eta(t)\geqslant0
\end{align*}
for almost every $t\in(0,T_\varepsilon)$. Therefore, \eqref{Weighted-Identity} can be written as
\begin{align}\label{G-Beta-F-J}
	G^\eta(t)+\beta^\eta(t)F^\eta(t)=J^\eta(t).
\end{align}

\begin{lemma}[Positivity and upper estimate for $F^\eta$]\label{Lemma-F-Estimate}
	Under the assumptions on the initial data, for every $t\in[0,T_\varepsilon)$,
	\begin{align}\label{F-Differential-Equation}
		(F^\eta)'(t)+\gamma^\eta(t)F^\eta(t)=J^\eta(t).
	\end{align}
	Solving this first-order equation yields
	\begin{align}\label{F-Representation}
		F^\eta(t)=\mathrm{e}^{-\int_0^t\gamma^\eta(\tau)\dd\tau}F^\eta(0)+\int_0^t\mathrm{e}^{-\int_s^t\gamma^\eta(\tau)\dd\tau}J^\eta(s)\dd s.
	\end{align}
	In particular, $F^\eta(t)>0$ for every $t>0$. Moreover,
	\begin{align}\label{F-Upper-Estimate}
		F^\eta(t)\leqslant\mathrm{e}^{-\frac{7\eta}{4}t}F^\eta(0)+\frac{4}{7\eta}J^\eta(t).
	\end{align}
\end{lemma}

\begin{proof}
	The support condition and the regularity of $u$ justify the differentiation of $F^\eta$. Using $\Psi_t^\eta=r^\eta\Psi^\eta$, we obtain
	\begin{align*}
		(F^\eta)'(t)&=G^\eta(t)+r^\eta(t)F^\eta(t)\\
		&=J^\eta(t)-\big(\beta^\eta(t)-r^\eta(t)\big)F^\eta(t).
	\end{align*}
	Since $\beta^\eta-r^\eta=b-2r^\eta=\gamma^\eta$ by \eqref{Beta-Gamma-Definition}, this proves \eqref{F-Differential-Equation}. The integrating factor formula gives
	\eqref{F-Representation}. Moreover,
	\begin{align*}
		F^\eta(0)=\varepsilon\int_{\mb{R}^n}u_0(x)\varphi^\eta(x)\dd x\geqslant0,
	\end{align*}
	whereas $J^\eta(t)>0$. Hence, \eqref{F-Representation} implies $F^\eta(t)>0$ for every $t>0$. Because $J^\eta$ is nondecreasing and $\gamma^\eta\geqslant\frac{7\eta}{4}$ by \eqref{Beta-Gamma-Bounds}, we deduce that
	\begin{align*}
		F^\eta(t)&\leqslant\mathrm{e}^{-\frac{7\eta}{4}t}F^\eta(0)+J^\eta(t)\int_0^t\mathrm{e}^{-\frac{7\eta}{4}(t-s)}\dd s\\
		&\leqslant\mathrm{e}^{-\frac{7\eta}{4}t}F^\eta(0)+\frac{4}{7\eta}J^\eta(t),
	\end{align*}
	which proves \eqref{F-Upper-Estimate}.
\end{proof}

\begin{remark}\label{Remark-Exact-Advantage}
	If one employs an adjoint subsolution instead of an exact solution, the differential relation for $F^\eta$ contains an additional integral term generated by the residual of the adjoint equation. Its sign can be exploited only after the positivity of $F^\eta$ has been established, which leads to the first-zero argument used in \cite{Hamza=2026,Hamza=2026-2}. In the present setting, the exact adjoint equation \eqref{Exact-Adjoint-Equation} eliminates this residual term, and the positivity of $F^\eta$ follows directly from \eqref{F-Representation}.
\end{remark}

\subsection{Coercive lower bound for the second weighted functional}

The relation \eqref{G-Beta-F-J} and the upper bound \eqref{F-Upper-Estimate} yield a coercive lower bound for $G^\eta$ in terms of the nondecreasing functional $J^\eta$. This direct comparison avoids differentiating $G^\eta$.

\begin{prop}[Coercive lower bound for $G^\eta$]\label{Prop-G-Estimate}
	For every $t\in[1,T_\varepsilon)$, one has
	\begin{align}\label{G-Coercive-Estimate}
		G^\eta(t)\geqslant\frac{1}{4}J^\eta(t).
	\end{align}
	Consequently,
	\begin{align}\label{G-Positive-Lower}
		G^\eta(t)\geqslant\frac{\varepsilon}{4}C^\eta(u_0,u_1)>0\ \ \text{for every}\ \ t\in[1,T_\varepsilon).
	\end{align}
\end{prop}

\begin{proof}
	By \eqref{G-Beta-F-J}, \eqref{F-Upper-Estimate} and \eqref{Beta-Gamma-Bounds}, we are able to derive
	\begin{align*}
		G^\eta(t)=J^\eta(t)-\beta^\eta(t)F^\eta(t)\geqslant\frac{9}{28}J^\eta(t)-\frac{19\eta}{16}\mathrm{e}^{-\frac{7\eta}{4}t}F^\eta(0).
	\end{align*}
	Moreover, by \eqref{Initial-Weighted-Constant} and \eqref{Beta-Gamma-Bounds}, one claims
	\begin{align*}
		C^\eta(u_0,u_1)\geqslant\beta^\eta(0)\int_{\mb{R}^n}u_0(x)\varphi^\eta(x)\dd x=\frac{\beta^\eta(0)}{\varepsilon}F^\eta(0)\geqslant\frac{13\eta}{16\varepsilon}F^\eta(0).
	\end{align*}
	Since $J^\eta(t)\geqslant\varepsilon C^\eta(u_0,u_1)$, it follows that
	\begin{align*}
		F^\eta(0)\leqslant\frac{16}{13\eta}J^\eta(t).
	\end{align*}
	Therefore,
	\begin{align*}
		G^\eta(t)\geqslant\left(\frac{9}{28}-\frac{19}{13}\mathrm{e}^{-\frac{7\eta}{4}t}\right)J^\eta(t).
	\end{align*}
	For $t\geqslant1$ and $\eta\geqslant2$, the elementary computation yields
	\begin{align*}
		\frac{9}{28}-\frac{19}{13}\mathrm{e}^{-\frac{7\eta}{4}t}\geqslant\frac{9}{28}-\frac{19}{13}\mathrm{e}^{-\frac{7}{2}}>\frac{1}{4}.
	\end{align*}
	This proves \eqref{G-Coercive-Estimate}. The lower bound \eqref{G-Positive-Lower} follows from $J^\eta(t)\geqslant\varepsilon C^\eta(u_0,u_1)$.
\end{proof}

\begin{remark}\label{Remark-Direct-G}
The estimate in Proposition~\ref{Prop-G-Estimate} follows directly from \eqref{G-Beta-F-J} and \eqref{F-Differential-Equation}, without differentiating $G^\eta$. Therefore, no derivatives of the coefficients are needed beyond those contained in the potential $V$.
\end{remark}

\section{Proof of the main results}\label{Section-Proof}

\subsection{Integral estimate for the adjoint test function}

We first estimate the weighted volume of the support of the solution.

\begin{lemma}[Weighted volume estimate]\label{Lemma-Weighted-Volume}
	For the adjoint test function $\Psi^\eta$ defined in \eqref{Adjoint-Test-Function}, there exists a constant $C_{\eta,R}>0$ such that
	\begin{align}\label{Weighted-Volume}
		\int_{|x|\leqslant R+t}\Psi^\eta(t,x)\dd x\leqslant C_{\eta,R}\,\mathrm{e}^{\frac{1}{2}B(t)}(1+t)^{\frac{n-1}{2}}
	\end{align}
	for every $t\geqslant0$.
\end{lemma}

\begin{proof}
	By \eqref{Temporal-Comparability}, one may deduce
	\begin{align*}
		\int_{|x|\leqslant R+t}\Psi^\eta(t,x)\dd x\lesssim\mathrm{e}^{-\eta t+\frac{1}{2}B(t)}\int_{|x|\leqslant R+t}\varphi^\eta(x)\dd x.
	\end{align*}
	For $n=1$, it follows directly from \eqref{Spatial-Test-Function} that
	\begin{align*}
		\int_{|x|\leqslant R+t}\varphi^\eta(x)\dd x\lesssim\mathrm{e}^{\eta(R+t)}.
	\end{align*}
	For $n\geqslant2$, \eqref{Spatial-Pointwise-Estimate} and polar coordinates give
	\begin{align*}
		\int_{|x|\leqslant R+t}\varphi^\eta(x)\dd x&\lesssim_\eta\int_0^{R+t}\mathrm{e}^{\eta r}(1+r)^{-\frac{n-1}{2}}r^{n-1}\dd r\\
		&\lesssim_\eta\int_0^{R+t}\mathrm{e}^{\eta r}(1+r)^{\frac{n-1}{2}}\dd r\\
		&\lesssim_{\eta,R}\mathrm{e}^{\eta t}(1+t)^{\frac{n-1}{2}}.
	\end{align*}
	Thus, for every $n\geqslant1$,
	\begin{align*}
		\int_{|x|\leqslant R+t}\varphi^\eta(x)\dd x\leqslant C_{\eta,R}\,\mathrm{e}^{\eta t}(1+t)^{\frac{n-1}{2}}.
	\end{align*}
	Combining the last estimate with \eqref{Temporal-Comparability} proves \eqref{Weighted-Volume}.
\end{proof}

\subsection{The nonlinear differential inequality}

To convert the coercive estimate for $G^\eta$ into a scalar differential inequality, for $t\geqslant1$, define
\begin{align}\label{Comparison-Functional}
	L^\eta(t):=\frac{1}{4}\left(\int_1^tN^\eta(s)\dd s+\varepsilon C^\eta(u_0,u_1)\right).
\end{align}
It follows directly from the definition that
\begin{align}\label{L-Derivative}
	(L^\eta)'(t)=\frac{1}{4}N^\eta(t)
\end{align}
for almost every $t\in(1,T_\varepsilon)$. Moreover, thanks to $L^\eta(t)\leqslant\frac{1}{4}J^\eta(t)$, \cref{Prop-G-Estimate} yields
\begin{align}\label{G-Dominates-L}
	G^\eta(t)\geqslant L^\eta(t)>0\ \ \text{for every}\ \ t\in[1,T_\varepsilon).
\end{align}

For almost every $t\in(1,T_\varepsilon)$, the positivity in \eqref{G-Dominates-L}, H\"older's inequality, and \eqref{Support-Time-Derivative} give
\begin{align*}
	G^\eta(t)&\leqslant\int_{\mb{R}^n}|u_t(t,x)|\Psi^\eta(t,x)\dd x\notag\\
	&\leqslant\big(N^\eta(t)\big)^{\frac{1}{p}}\left(\int_{|x|\leqslant R+t}\Psi^\eta(t,x)\dd x\right)^{\frac{1}{p'}}.
\end{align*}
It follows from \cref{Lemma-Weighted-Volume} that
\begin{align}\label{Nonlinearity-Lower-G}
	N^\eta(t)&\geqslant\big(G^\eta(t)\big)^p\left(\int_{|x|\leqslant R+t}\Psi^\eta(t,x)\dd x\right)^{1-p}\notag\\
	&\geqslant C\big(G^\eta(t)\big)^p\mathrm{e}^{-\frac{p-1}{2}B(t)}(1+t)^{-\frac{(n-1)(p-1)}{2}}.
\end{align}
Combining \eqref{G-Dominates-L}, \eqref{L-Derivative}, and \eqref{Nonlinearity-Lower-G}, we obtain the scalar differential inequality that drives the blow-up argument.

\begin{prop}[Nonlinear differential inequality]\label{Prop-Nonlinear-ODE}
	For almost every $t\in(1,T_\varepsilon)$, one has
	\begin{align}\label{Nonlinear-ODE}
		(L^\eta)'(t)\geqslant C\big(L^\eta(t)\big)^p\mathrm{e}^{-\frac{p-1}{2}B(t)}(1+t)^{-\frac{(n-1)(p-1)}{2}}.
	\end{align}
	The constant $C>0$ is independent of $\varepsilon$ and $t$.
\end{prop}

\subsection{Proof of the general blow-up criterion}

\begin{proof}[Proof of Theorem~\ref{Thm-General-Criterion}]
	Fix $\eta\geqslant\eta_0$ once and for all. Since $\ml{H}_p^{-1}(y)\geqslant1$ for every $y\geqslant0$, \eqref{General-Lifespan-Bound} is immediate when $T_\varepsilon\leqslant1$. Hence, we may assume that $T_\varepsilon>1$. Since $L^\eta(t)>0$, \eqref{Nonlinear-ODE} implies
	\begin{align*}
		\frac{\dd}{\dd t}[L^\eta(t)]^{-(p-1)}\leqslant-C(p-1)\,\mathrm{e}^{-\frac{p-1}{2}B(t)}(1+t)^{-\frac{(n-1)(p-1)}{2}}
	\end{align*}
	for almost every $t\in(1,T_\varepsilon)$. Integrating over $[1,t]$, we obtain
	\begin{align*}
		[L^\eta(t)]^{-(p-1)}\leqslant [L^\eta(1)]^{-(p-1)} -C(p-1)\ml{H}_p(t).
	\end{align*}
	Moreover, \eqref{Comparison-Functional} gives $L^\eta(1)=\frac{\varepsilon}{4}C^\eta(u_0,u_1)$. 
	Since $C^\eta(u_0,u_1)>0$ is independent of $\varepsilon$, the positivity of the left-hand side yields the necessary condition
	\begin{align}\label{Necessary-H-Condition}
		\ml{H}_p(t)\leqslant C\varepsilon^{-(p-1)}\ \ \text{for every}\ \ t\in[1,T_\varepsilon).
	\end{align}
Suppose, by contradiction, that $T_\varepsilon=\infty$. Then \eqref{Necessary-H-Condition} holds for every $t\geqslant1$, which contradicts \eqref{Divergence-Condition}. Therefore, $T_\varepsilon<\infty$.

It remains to prove \eqref{General-Lifespan-Bound}. Letting $t\uparrow T_\varepsilon$ in \eqref{Necessary-H-Condition}, we obtain
\begin{align*}
	\ml{H}_p(T_\varepsilon)\leqslant C\varepsilon^{-(p-1)}.
\end{align*}
Applying the increasing function $\ml{H}_p^{-1}$ to both sides gives
\begin{align*}
	T_\varepsilon\leqslant\ml{H}_p^{-1}\big(C\varepsilon^{-(p-1)}\big),
\end{align*}
which proves \eqref{General-Lifespan-Bound} and completes the proof.
\end{proof}

\subsection{Explicit subcritical and critical estimates}

\begin{proof}[Proof of Corollary~\ref{Coro-Log-Damping}]
	The desired estimates are immediate when $T_\varepsilon\leqslant1$ after enlarging the constants. We may therefore assume that $T_\varepsilon>1$. It follows from \eqref{Log-Damping-Assumption} and the definition of $\ml{H}_p$ that
	\begin{align}\label{H-Power-Lower}
		\ml{H}_p(t)\geqslant C\int_1^t(1+s)^{-\alpha}\dd s\ \ \text{with}\ \ \alpha:=\frac{(n+\mu-1)(p-1)}{2}.
	\end{align}
	 Let first $n+\mu>1$ and $1<p<p_{\mathrm{Gla}}(n+\mu)$. In this case, $\alpha<1$, and
	\begin{align*}
		\ml{H}_p(t)\geqslant C\big((1+t)^{1-\alpha}-1\big).
	\end{align*}
	By Theorem~\ref{Thm-General-Criterion} and \eqref{Necessary-H-Condition}, we know
	\begin{align*}
		(1+T_\varepsilon)^{1-\alpha}\leqslant C\big(1+\varepsilon^{-(p-1)}\big)\leqslant C\varepsilon^{-(p-1)}
	\end{align*}
	for every sufficiently small $\varepsilon>0$. According to the following identity:
	\begin{align*}
		\frac{p-1}{1-\alpha}=\frac{2(p-1)}{2-(n+\mu-1)(p-1)},
	\end{align*}
	the subcritical estimate in \eqref{Explicit-Lifespan-Bounds} follows. If $p=p_{\mathrm{Gla}}(n+\mu)$, then $\alpha=1$, and \eqref{H-Power-Lower} gives
	\begin{align*}
		\ml{H}_p(t) \geqslant C\log\left(\frac{1+t}{2}\right).
	\end{align*}
	Therefore,
	\begin{align*}
		\log\left(\frac{1+T_\varepsilon}{2}\right)\leqslant C\varepsilon^{-(p-1)},
	\end{align*}
	which yields the critical estimate in \eqref{Explicit-Lifespan-Bounds}. Finally, if $n=1$ and $\mu=0$, then \eqref{H-Power-Lower} becomes
	\begin{align*}
		\ml{H}_p(t)\geqslant C(t-1).
	\end{align*}
	It leads to
	\begin{align*}
		T_\varepsilon\leqslant1+C\varepsilon^{-(p-1)}\leqslant C\varepsilon^{-(p-1)}
	\end{align*}
	for every sufficiently small $\varepsilon>0$. This proves \eqref{One-Dimensional-Zero-Shift}.
\end{proof}

\begin{remark}\label{Remark-Critical-Mechanism}
	At $p=p_{\mathrm{Gla}}(n+\mu)$, the assumption \eqref{Log-Damping-Assumption} provides the lower bound
	\begin{align*}
		\mathrm{e}^{-\frac{p-1}{2}B(t)}(1+t)^{-\frac{(n-1)(p-1)}{2}}\geqslant C(1+t)^{-1}.
	\end{align*}
	The logarithmic divergence of this lower bound is sufficient to cover the critical endpoint. No logarithmic modification of the adjoint profile is required, since the Volterra equation constructs an exact
	profile comparable to $\mathrm{e}^{-\eta t+\frac{1}{2}B(t)}$. For the scale-invariant coefficients, the time weight is exactly $(1+t)^{-1}$, while the leading order of the temporal adjoint profile is $(1+t)^{\frac{\mu}{2}}\mathrm{e}^{-\eta t}$.
\end{remark}

\begin{proof}[Proof of Corollary~\ref{Coro-Scale-Invariant}]
	The computation in \cref{Example-Scale-Invariant} shows that all the assumptions of Corollary~\ref{Coro-Log-Damping} are satisfied with $C_B=0$, independently of the sign of $\delta$. The conclusion follows directly.
\end{proof}

\section{Concluding remarks and open problems}
\label{Sec-Concluding-Remarks}

In this paper, the quantities $B$ and $V$ enter the proof in different ways. The integrability of $V$ allows us to construct a positive adjoint solution, while the growth of $B$ determines the time weight in the blow-up estimate. Hence, the argument also applies when the damping or mass coefficient is non-integrable or nondecaying, provided that their combined contribution to $V$ remains integrable.

A natural question is whether the divergence condition for $\ml{H}_p$ in Theorem~\ref{Thm-General-Criterion} is sharp under
suitable assumptions on the associated linear evolution. The logarithmically corrected model in Example~\ref{Example-Logarithmic-Correction} is particularly relevant. At the shifted Glassey exponent, the present criterion detects a transition at
\begin{align*}
	\kappa=n+\mu-1,
\end{align*}
including a double exponential lifespan upper bound at the borderline. It remains open whether this transition describes the actual nonlinear threshold or only a limitation of the present test function method.

Another question concerns the optimality of the short-range condition
\begin{align*}
	V\in L^1(\Rplusast).
\end{align*}
This assumption allows us to construct a bounded positive amplitude satisfying
\begin{align*}
	a_\eta(t)\to1
	\quad\text{and}\quad
	a_\eta'(t)\to0
	\ \ \text{as}\ \ t\to\infty.
\end{align*}
A basic borderline long-range model lying beyond the present framework is
\begin{align*}
	V(t)=\frac{q_1}{1+t}
	\ \ \text{with}\ \ 
	q_1\in\mb{R}\setminus\{0\}.
\end{align*}
We emphasize that this decay is long-range rather than scale-invariant, since the scale-invariant order for the effective potential is $(1+t)^{-2}$. For this model, the Volterra construction posed from infinity with the normalization $a_\eta(t)\to1$ is no longer directly applicable. A formal dominant-balance argument suggests instead that
\begin{align*}
	a_\eta(t)\sim(1+t)^{\frac{q_1}{2\eta}}
	\ \ \text{as}\ \ t\to\infty.
\end{align*}
It remains open whether such a correction changes the sharp blow-up threshold or the lifespan asymptotics. We plan to investigate the precise influence of $q_1$ on the blow-up range in future work.

\section*{Acknowledgments}
Wenhui Chen is supported in part by the National Natural Science Foundation of China (grant No. 12301270) and the Guangdong Basic and Applied Basic Research Foundation (grant No. 2025A1515010240).

\end{document}